\documentclass[english,reqno]{amsart}
\usepackage[T1]{fontenc}
\usepackage[latin9]{inputenc}
\usepackage{bbm}
\usepackage{geometry}
\usepackage{amstext}
\usepackage{amsthm}
\usepackage{amssymb}
\usepackage{mathtools, color,hyperref}

\numberwithin{equation}{section}
\numberwithin{figure}{section}
\theoremstyle{plain}
\newtheorem{thm}{\protect\theoremname}
  \theoremstyle{plain}
  \newtheorem{lem}[thm]{\protect\lemmaname}

\allowdisplaybreaks

\usepackage{babel}
\providecommand{\lemmaname}{Lemma}
\providecommand{\theoremname}{Theorem}
\newcommand{\x}{\mathbf{x}}
\newcommand{\y}{\mathbf{y}}

\newtheorem{remark}{Remark}
\newtheorem{corollary}{Corollary}

\begin{document}

\title{Lower Lipschitz Bounds for Phase Retrieval from Locally Supported Measurements}\thanks{M.A. Iwen:  Department of Mathematics and Department of CMSE, Michigan State University ({\tt markiwen@math.msu.edu}).  M.A. Iwen was supported in part by NSF DMS-1416752. \\ 
\indent Sami E. Merhi: Department of Mathematics, Michigan State University ({\tt merhisam@msu.edu}) \\ 
\indent Michael Perlmutter: Department of Computational Mathematics Science and Engineering (CMSE), Michigan State University ({\tt perlmut6@msu.edu}).}

\author{Mark A. Iwen, Sami Merhi, Michael Perlmutter}


\begin{abstract}
In this short note, we consider the worst case noise robustness of any phase retrieval algorithm which aims to reconstruct all nonvanishing vectors $\mathbf{x} \in \mathbb{C}^d$ (up to a single global phase multiple) from the magnitudes of an arbitrary collection of local correlation measurements. Examples of such measurements include both spectrogram measurements of $\mathbf{x}$ using locally supported windows and masked Fourier transform intensity measurements of $\mathbf{x}$ using bandlimited masks.  As a result, the robustness results considered herein apply to a wide range of both ptychographic and Fourier ptychographic imaging scenarios.  In particular, the main results imply that the accurate recovery of  high-resolution images of extremely large samples using highly localized probes is likely to require an extremely large number of measurements in order to be robust to worst case measurement noise, independent of the recovery algorithm employed.  Furthermore, recent pushes to achieve high-speed and high-resolution ptychographic imaging of integrated circuits for process verification and failure analysis will likely need to carefully balance probe design (e.g., their effective time-frequency support) against the total number of measurements acquired in order for their imaging techniques to be stable to measurement noise, no matter what reconstruction algorithms are applied.
\end{abstract}

\maketitle

\section{Introduction and Statement of Results}

We consider the robustness of the {\it finite-dimensional phase retrieval problem} in which one attempts to recover a signal $\mathbf{x}\coloneqq\left(\mathbf{x}(1),\ldots,\mathbf{x}(d)\right)^T\in\mathbb{C}^d$ from one of two nonlinear measurement maps $\alpha, \beta:\mathbb{C}^d\rightarrow \mathbb{R}^N$ given by
\begin{equation*}
\alpha(\mathbf{x}) = \{|\langle \mathbf{x}, \mathbf{f_k}\rangle|\}_{k=1}^N \text{ and } \beta(\mathbf{x}) = \{|\langle \mathbf{x},\mathbf{f_k}\rangle|^2\}_{k=1}^N,
\end{equation*}  
where the vectors $\{\mathbf{f_1},\ldots, \mathbf{f_N}\} \subset \mathbb{C}^d$ form a frame (i.e., a spanning set) of $\mathbb{C}^d$.  This problem is motivated by inverse problems that arise in several scientific areas including optics \cite{Walther1963}, astronomy \cite{ Fienup1987}, quantum mechanics \cite{Corbett2006}, and audio signal processing \cite{nawab1983signal,sturmel2011signal}.  In particular, we will focus on a special class of frame vectors $\mathbf{f_k}$ which have {\it localized support} (i.e., all of whose nonzero entries are contained in an interval of length at most $\delta \ll d$).  Such frames are commonly encountered in applications like ptychographic imaging in which small overlapping regions of a much larger specimen are illuminated one at a time, and a detector captures the intensities of the resulting local diffraction patterns \cite{Rodenburg2008}.

It is clear that for any $\theta\in\mathbb{R}$ one has $\alpha(\mathbbm{e}^{i\theta}\mathbf{x})=\alpha(\mathbf{x})$ and $\beta(\mathbbm{e}^{i\theta}\mathbf{x})=\beta(\mathbf{x}).$ Therefore, we can at best hope to recover $\mathbf{x}$ up to the equivalence relation $\mathbf{x}\sim \mathbf{x}',$ if  $\mathbf{x}=\mathbbm{e}^{i\theta}\mathbf{x}'$ for some $\theta\in\mathbb{R}.$ Following the work of Balan et al. \cite{balan2016frames,Balan2016}, we will consider two commonly used metrics on $\mathbb{C}^d/\sim:$ the natural metric 
\begin{equation*}
D_2(\mathbf{x},\mathbf{x}') \coloneqq \min_{\theta\in\mathbb{R}} \|\mathbf{x}-\mathbbm{e}^{\mathbbm{i}\theta}\mathbf{x}'\|_2,
\end{equation*}
and the matrix-norm induced metric 
\begin{equation*}
d_1(\mathbf{x},\mathbf{x}') \coloneqq \|\mathbf{x}\mathbf{x}^*-\mathbf{x}'\mathbf{x}'^*\|_1 \coloneqq \sum_k \sigma_k(\mathbf{x}\mathbf{x}^*-\mathbf{x}'\mathbf{x}'^*),
\end{equation*}
where  $\sigma_k(\mathbf{x}\mathbf{x}^*-\mathbf{x}'\mathbf{x}'^*)$ is the $k$-th singular value of the (at most rank-two) matrix $\mathbf{x}\mathbf{x}^*-\mathbf{x}'\mathbf{x}'^*$.  In \cite{Balan2016}, Balan et al. showed that if $\alpha$ and $\beta$ are injective on $\mathbb{C}^d/\sim,$ then $\beta$ is bi-Lipschitz with respect to $d_1,$ and $\alpha$ is bi-Lipschitz with respect to $D_2,$ where in both cases $\mathbb{R}^N$ is equipped with the Euclidean norm.

Motivated by applications such as (Fourier) ptychography \cite{Rodenburg2008,zheng2013wide} and related numerical methods \cite{iwen2016fpr,robustPR}, we will study frames which are constructed as the shifts of a family of locally supported measurement vectors. Specifically, we assume that $\{\mathbf{m_1},\mathbf{m_2},\ldots,\mathbf{m_K}\}$ is a family of measurement masks in $\mathbb{C}^d$ such that for all $1\leq k \leq K$ the nonzero entries of $\mathbf{m_k}$ are contained in the set $[\delta] := \{ 1, \dots, \delta \}$ for some $\delta\leq \frac{d}{4}$ (although all of our results remain valid if the support of our masks are contained in any interval of length $\delta$).  Letting $L$ be an integer which divides $d$, such that $a\coloneqq\frac{d}{L}<\delta,$ we consider nonlinear phaseless measurement maps $Y,Z: \mathbb{C}^d\rightarrow \mathbb{R}^{K\times L}$ defined by their coordinate functions
\begin{equation}
Y_{k,\ell}(\mathbf{x}) \coloneqq |\langle S_{\ell a}\mathbf{m_k}, \mathbf{x}\rangle|^2
\label{equ:Ymap}
\end{equation} 
and
\begin{equation}
Z_{k,\ell}(\mathbf{x})\coloneqq|\langle S_{\ell a}\mathbf{m_k}, \mathbf{x}\rangle|
\label{equ:Zmap}
\end{equation}
for $1\leq k \leq K$ and $1\leq \ell \leq L$.
Here $S_\ell$ is  the circular shift operator on $\mathbb{C}^d$ defined for all $\ell\in\mathbb{Z}$ by
\begin{equation*}
(S_\ell\mathbf{x})(n) \coloneqq \x\left((n+\ell-1) \!\!\!\!\mod d+1\right).
\end{equation*} 
(The $+\:1$ is needed because we are indexing our vectors from one.) 
 For notational convenience, we will assume that $d$ is even, although our results remain valid, with similar proofs, when $d$ is odd. 

The purpose of this paper is to provide lower bounds on the Lipschitz constants of any maps, $A$ and $B,$ which reconstruct $\mathbf{x}$ from $Y$ and $Z$, respectively. 
With such lower bounds in hand, one would be better equipped to, e.g., judge the optimality of theoretical noisy reconstruction guarantees for phase retrieval algorithms which utilize locally supported measurements (see, e.g., \cite{iwen2016fpr,robustPR}).
Unfortunately, $Y$ and $Z$ are not injective on all of $\mathbb{C}^d/\sim$.  For example, if two vectors $\mathbf{x}^{\pm} \in \mathbb{C}^d$ are defined by 
\begin{equation}
\x^{\pm}(n) \coloneqq \begin{cases} 
      1, & 1\leq n \leq \frac{d}{2}-\delta \\
      0, & \frac{d}{2}-\delta < n \leq \frac{d}{2} \\
      \pm 1,& \frac{d}{2}<n\leq d-\delta \\
	  0, & d-\delta<n\leq d
   \end{cases},
   \label{equ:AtollVec}
\end{equation} 
then $\mathbf{x^+}\not\sim\mathbf{x^-},$ but $Y(\mathbf{x}^+)=Y(\mathbf{x}^-)$ and $Z(\mathbf{x}^+)=Z(\mathbf{x}^-).$ (If $d$ were odd, we could add an extra entry of $1$ to  $\mathbf{x}^\pm.$) However, it can be shown \cite{iwen2016fpr} that $Y$ and $Z$ are injective when restricted to the subset of $\mathbb{C}^d$ such that $\x(n)\neq 0$ for all $1\leq n \leq d$, for certain choices of masks in the case where $L=d$.  Given this, we will consider the maps $Y$ and $Z$  restricted to the subset 
\begin{equation*}
\mathcal{C}_{p,q}=\{\mathbf{x}\in\mathbb{C}^d/\sim\text{ such that } p\leq |\x(n)| \leq q \text{ for all }1\leq n \leq d\},
\end{equation*}
 for some fixed $0<p\leq q,$  and provide lower bounds on the Lipschitz constants of $A$ and $B$ which grow linearly with respect to the ratio $\frac{q}{p}.$ 


\subsection{Related Work and Implications}


Our local measurement maps \eqref{equ:Ymap} and \eqref{equ:Zmap} are closely related to several practical measurement models that have been explored in the phase retrieval literature including, for example, Short-Time Fourier Transform (STFT) magnitude measurements (see, e.g., \cite{bendory2018non,jaganathan2016stft,pfander2016robust,salanevich2015polarization}).  In particular, suppose that our STFT magnitude measurements are generated by a compactly supported window $\mathbf{w} \in \mathbb{C}^d$ whose $n^{\rm th}$-entry $\mathbf{w}(n)$ is nonzero only if $n \in [\delta]$.  In this setting, we can use one locally supported mask $\mathbf{m_k}$ to represent each measured frequency $\omega_k \in \Omega \subset [d] := \{ 1, \dots, d \}$ by letting $\mathbf{m_k} := W_{\omega_k} \mathbf{w}$ for each frequency index $k$, where $W_{\omega_k}$ is the modulation operator defined on $\mathbb{C}^d$ by 
\begin{equation*}
(W_{\omega_k} \mathbf{w})(n) :=  \mathbbm{e}^{\frac{2 \pi \mathbbm{i} (n-1) (\omega_k - 1)}{d}}\mathbf{w}(n).\end{equation*}  In this case, we have
$$\left| \left\langle S_{\ell a}\mathbf{m_k}, \mathbf{x} \right \rangle \right| ~=~ \left| \left\langle \mathbf{x}, S_{\ell a} W_{\omega_k} \mathbf{w} \right\rangle \right| ~=~ \left| \left\langle \mathbf{x}, \mathbbm{e}^{\frac{2 \pi \mathbbm{i} \ell a (\omega_k-1)}{d}} W_{\omega_k} S_{\ell a} \mathbf{w} \right\rangle \right|~=~ \left| \left\langle \mathbf{x}, W_{\omega_k} S_{\ell a} \mathbf{w} \right\rangle \right|$$
for all $k$ and $\ell.$
Therefore, one can see that the main results below yield lower Lipschitz bounds for any such STFT magnitude measurements in terms of the total number of shifts $L$, the number of measured frequencies $K$, and the window $\mathbf{w}$'s support size $\delta.$ 

Another common model considered in the phase retrieval literature concerns masked Fourier measurements of the form 
\begin{equation}
|F ~{\rm Diag}(\mathbf{w_k}) ~\mathbf{x} |^2,
\label{equ:MaskedFouerierMeas}
\end{equation}
where $F$ is the $d \times d$ discrete Fourier transform matrix whose entries are defined by 
\begin{equation*}
F_{j,k} := 
\mathbbm{e}^{-2\pi\mathbbm{i}\frac{(j-1)(k-1)}{d}},
\end{equation*} 
and  $\{\mathbf{w_1},\ldots,\mathbf{w_k}\} \subset \mathbb{C}^d$ is a family of measurement vectors (see, e.g., \cite{bandeira2014phase,candes2015phase,Candes2014WF,gross2017improved}).  In this setting one can ask what effect, if any, requiring each $\mathbf{w_k}$ to be bandlimited (i.e., to have support size $\delta \ll d$ in the Fourier basis) might have on the stability of these measurements.  Furthermore, one might also consider subsampling each of the masked Fourier measurements in frequency instead of acquiring measurements for all $d$ frequencies. (This may even be a necessity due to, for example, detector limitations.)  We will show that our results may also be applied to these types of measurements as a special case.

Suppose for example that each measurement vector $\mathbf{w_k}$ has $\widehat{\mathbf{w_k}}(n)\coloneqq (F\mathbf{w_k})(n) = 0$ for all $n \notin \{ 1 \} \cup \{ d - \delta + 2, \dots, d \}$.\footnote{Note that this particular support interval (modulo d) is not particularly special.  The same arguments below can be extended to apply to any interval of support of size $\leq \delta$ in a straightforward fashion.} For a vector $\mathbf{u} \in \mathbb{C}^d$, let $\tilde{\mathbf u} \in \mathbb{C}^d$ be the vector obtained by reflecting the entries of $\mathbf{u}$ about its first entry so that
\begin{equation*}
\tilde{\mathbf u}(n) := \mathbf{u}\left((1-n)\!\!\!\!\mod d+1\right).
\end{equation*}
In this case, we see that the measurements \eqref{equ:MaskedFouerierMeas} are given by the measurement map \eqref{equ:Ymap} applied to $\widehat{{\bf x}}$ with the locally supported measurement masks $\mathbf{m_k} := \frac{1}{d} \overline{\widetilde{\widehat{{\bf w_{k}}}}}$.  
Indeed, 
\begin{align}
\left| \left\langle S_{\ell a}\mathbf{m_k}, \widehat{{\bf x}} \right \rangle \right|  &=  \frac{1}{d}\left|\left\langle \widehat{{\bf x}},S_{\ell a}\overline{\widetilde{\widehat{{\bf w_{k}}}}}\right\rangle \right| =  \frac{1}{d}\left|\sum_{n=1}^{d}\widehat{{\bf x}}\left(n\right)S_{\ell a}\widetilde{\widehat{{\bf w_{k}}}}\left(n\right)\right| \nonumber\\
&=  \frac{1}{d}\left|\sum_{n=1}^{d}\widehat{{\bf x}}\left(n\right)\widehat{{\bf w_{k}}}\left(\left(1-\ell a-n \right) \!\!\!\!\mod d+1\right)\right| \nonumber\\
& =  \frac{1}{d}\left|\left(\widehat{{\bf w_{k}}} * \widehat{{\bf x}}\right)\left(-\ell a\!\!\!\!\mod d + 1\right)\right|, \label{eq:MaskedFourierCalc}
\end{align}
where $*$ is circular convolution given by
\begin{equation*}
(\x * \y)(m) = \sum_{n=1}^d \x(n)\y((m-n)\!\!\!\!\mod d+1).
\end{equation*}
  Continuing from \eqref{eq:MaskedFourierCalc}, we see by the convolution theorem
$$\left| \left\langle S_{\ell a}\mathbf{m_k}, \widehat{{\bf x}} \right \rangle \right|~=~ \left|F\left({\bf w_{k}}\circ{\bf x}\right)\left((-\ell a\!\!\!\!\mod d) + 1\right)\right| ~=~ \left|F\left(\mbox{Diag}\left({\bf w_{k}}\right){\bf x}\right)\left((-\ell a\!\!\!\!\mod d) + 1\right)\right|,$$
where $\circ$ represents the Hadamard (componentwise) product.

As a result, we see that recovering a vector ${\bf x}$ from masked Fourier measurements of the form \eqref{equ:MaskedFouerierMeas} with bandlimited measurement vectors $\mathbf{w_k}$ is equivalent to recovering $\widehat{{\bf x}}$ from measurements \eqref{equ:Ymap} with locally supported measurement masks $\mathbf{m_k}$.  Therefore, the main results below also yield lower Lipschitz bounds for any such masked Fourier magnitude measurements in terms of the total number of frequencies $L$ collected per measurement vector, the total number $K$ of measurement vectors used, and the maximum Fourier support size $\delta$ of each bandlimited measurement vector.

\subsection{Main Results}

The main results of this paper are the following two theorems which provide lower bounds for the Lipschitz constants of any maps $A$ and $B$ for which $A(Y(\mathbf{x}))=\mathbf{x}$ and $B(Z(\mathbf{x}))=\mathbf{x}$ for all $\mathbf{x} \in \mathcal{C}_{p,q}$. 
\begin{thm}\label{maintheorem2}
Let $0<p\leq q,$ and consider the map $Z$ restricted to the subset $\mathcal{C}_{p,q}\subset \mathbb{C}^{d}/\sim.$ 
Assume that $\delta\leq\frac{d}{4}$ and that $d=aL$ for some integer $1 \leq a<\delta.$ Then if $B$ is any Lipschitz map (with respect to $D_2$)  such that $B(Z(\mathbf{x}))=\mathbf{x}$ for all $\mathbf{x} \in \mathcal{C}_{p,q},$ we have that 
 \begin{equation}\label{Zclaim}
 C_B \geq C\frac{q\sqrt{da}}{p\sqrt{K}\|\mathbf{m}\|_\infty\delta^{3/2}}= C\frac{qd}{p\sqrt{KL}\|\mathbf{m}\|_\infty\delta^{3/2}} ,
\end{equation}
where $C_B$ is the Lipschitz constant of $B,$ $\|\mathbf{m}\|_\infty \coloneqq \max_{1\leq k \leq K}\|\mathbf{m_k}\|_\infty,$ and $C$ is a universal constant.
\end{thm}

\begin{thm}\label{maintheorem1}
Let $0<p\leq q,$ and consider the map $Y$ restricted to the subset  $\mathcal{C}_{p,q} \subset \mathbb{C}^d/\sim$. 
Assume that $\delta\leq\frac{d}{4}$ and that $d=aL$ for some integer $1 \leq a<\delta.$ Then if  $A$ is any Lipschitz map (with respect to $d_1$)  such that  $A(Y(\mathbf{x}))=\mathbf{x}$ for all $\mathbf{x} \in \mathcal{C}_{p,q},$ we have that 
 \begin{equation}\label{Yclaim}
C_A\geq C\frac{qd\sqrt{a}}{p\sqrt{K}\|\mathbf{m}\|^2_\infty \delta^{5/2}} =C\frac{qd^{3/2}}{p\sqrt{KL}\|\mathbf{m}\|^2_\infty \delta^{5/2}},
\end{equation}
where $C_A$ is the Lipschitz constant of $A,$ $\|\mathbf{m}\|_\infty \coloneqq \max_{1\leq k \leq K}\|\mathbf{m_k}\|_\infty,$ and $C$ is a universal constant. 
\end{thm}

Ideally, we would like a stable phase retrieval algorithm to have have $C_A = \mathcal{O}(1)$ (or $C_B = \mathcal{O}(1)$) while using only $KL = \mathcal{O}(d)$ total measurements.  Unfortunately, Theorems~\ref{maintheorem2}~and~\ref{maintheorem1} demonstrate that this is impossible when $\delta$, the support size of the masks, is very small.  At best, a phase retrieval algorithm that uses only $KL = \mathcal{O}(d)$ local correlation measurements can have global Lipschitz constants that are of size $\mathcal{O}\left(\frac{d}{\delta^{5/2}} \right)$ in the case of $Y$-measurements, and $\mathcal{O} \left( \frac{\sqrt{d}}{\delta^{3/2}} \right)$ in the case of $Z$-measurements.  This implies that extremely large samples $\mathbf{x}$ (i.e., with $d$ large) cannot be stably recovered from measurements which are noisy and extremely localized (i.e., with $\delta$ small)  in the worst case using only $\mathcal{O}(d)$ total measurements.  To contextualize this in an application setting, one may consider recent research initiatives aimed at achieving the ability to rapidly obtain detailed images of relatively large circuit boards \cite{IARPA_BAA}.  One approach to solving this problem involves using ptychographic imaging and taking STFT magnitude measurements of the circuit board using a probe (i.e., an STFT window function) with a comparably small effective support size $\delta$.  In this context, Theorem~\ref{maintheorem1} implies that the probe's effective support size should not be taken to be too small unless additional measurements are taken in order to help ensure stability to noise.

As we shall see, the proofs of both Theorems~\ref{maintheorem2}~and~\ref{maintheorem1}  will depend on signals modeled along the lines of \eqref{equ:AtollVec} whose support sets are composed of  two disjoint components separated from one another by at least $\delta$ zeroes.  In \cite{iwen2016fpr} it was noted that phase retrieval of such signals using locally supported masks $\mathbf{m_k}$ of the type proposed herein was impossible, and that recovery of signals with more than $\delta$ consecutive small entries appeared to be unstable.  Interestingly enough, subsequent work in the infinite-dimensional setting has independently identified such disjointly supported signals as being the principal cause of instability in phase retrieval problems using continuous Gabor measurements as well because they lead to measurements which are supported on disjoint subsets of the time-frequency plane \cite{Alaifari2018,Grohs2018}.  Similarly, we will use (essentially) disjointly supported signals similar to those in \eqref{equ:AtollVec} to provide lower bounds on the Lipschitz constants of our maps $A$ and $B$ using the fact that they $(i)$ are relatively far apart with respect to the $D_2$ and $d_1$ metrics defined above and $(ii)$ produce measurements with respect to our maps $Y$ and $Z$  which are (nearly) identical. 

\section{The Proofs of Theorem \ref{maintheorem2} and Theorem \ref{maintheorem1}}

We are now prepared to prove our main results.

\begin{proof}[Proof of Theorem \ref{maintheorem2}]
 First observe that for  any $\mathbf{x},\mathbf{x'}\in\mathcal{C}_{p,q},$
\begin{equation*}
D_2(\mathbf{x},\mathbf{x'}) = D_2(B(Z(\mathbf{x})),B(Z(\mathbf{x'})))\leq C_B\|Z(\mathbf{x})-Z(\mathbf{x'})\|_2.
\end{equation*}
Therefore,
\begin{equation}\label{ratioB}
C_B \geq \sup\frac{D_2(\mathbf{x},\mathbf{x'})}{\|Z(\mathbf{x})-Z(\mathbf{x'})\|_2},
\end{equation}
where the supremum is taken over all $\mathbf{x}\not\sim\mathbf{x'}\in\mathcal{C}_{p,q}.$ 
Define  $\mathbf{x}^+$ and $\mathbf{x}^-\in \mathbb{C}^d$ by 
\begin{equation*}
\mathbf{x}^{\pm}(n) \coloneqq \begin{cases} 
      q, & 1\leq n \leq \frac{d}{2}-\delta \\
      p, & \frac{d}{2}-\delta < n \leq \frac{d}{2} \\
      \pm q,& \frac{d}{2}<n\leq d-\delta \\
	  p, & d-\delta<n\leq d
   \end{cases}.
\end{equation*} 
Note that $D_2(\mathbf{x}^+,\mathbf{x}^-) \geq q\sqrt{d}$ 
since $d<\frac{d}{4}$ and for all $\theta\in\mathbb{R},$
\begin{align*}
\|\mathbf{x}^+-\mathbbm{e}^{\mathbbm{i}\theta}\mathbf{x}^-\|_2^2&\geq \sum_{n=1}^{d/2-\delta}|(1-\mathbbm{e}^{\mathbbm{i}\theta})q|^2 + \sum_{n=d/2+1}^{d-\delta}|(1+\mathbbm{e}^{\mathbbm{i}\theta})q|^2\\
 &= \left(\frac{d}{2}-\delta\right)q^2|1-\mathbbm{e}^{\mathbbm{i}\theta}|^2+\left(\frac{d}{2}-\delta\right)q^2|1+\mathbbm{e}^{\mathbbm{i}\theta}|^2\\
&\geq \frac{d}{4} q^2\left(|1-\mathbbm{e}^{\mathbbm{i}\theta}|^2+ |1+\mathbbm{e}^{\mathbbm{i}\theta}|^2\right) = dq^2,
\end{align*}
since $ |1-\mathbbm{e}^{\mathbbm{i}\theta}|^2+ |1+\mathbbm{e}^{\mathbbm{i}\theta}|^2=4$ for all $\theta.$ 
Let $Z^\pm\coloneqq Z(\mathbf{x^\pm}).$ 
We will show that 
\begin{equation}\label{Zdiff}
\|Z^+-Z^-\|_2\leq C\sqrt{K}p\|\mathbf{m}\|_\infty\frac{\delta^{3/2}}{\sqrt{a}}.
\end{equation}
 Since $B(Z^\pm)=\mathbf{x^\pm},$ combining this with (\ref{ratioB}) will  complete the proof.

Observe that for all $k,$ the support of $S_{\ell a}\mathbf{m_k}$ is contained in $[1+\ell a,\delta+\ell a].$ Therefore, $Z^+_{k,\ell} =   Z^-_{k,\ell}$ except when  $1+\ell a\leq \frac{d}{2} < \delta+\ell a$ or $1+\ell a \leq d-\delta <\delta+\ell a$ since if the support of $S_{\ell a}\mathbf{m_k}$ does not intersect $(\frac{d}{2},d-\delta],$ we have that $\langle S_{\ell a} \mathbf{m_k},\mathbf{x^+}\rangle = \langle S_{\ell a} \mathbf{m_k},\mathbf{x^-}\rangle,$ and if the support of $S_{\ell a}\mathbf{m_k}$ is contained in $(\frac{d}{2},d-\delta],$ then $\langle S_{\ell a}\mathbf{m_k},\mathbf{x^+}\rangle = -\langle S_{\ell a} \mathbf{m_k},\mathbf{x^-}\rangle.$ We will restrict attention to the case where $1+\ell a \leq d-\delta <\delta+\ell a.$ The case where $1+\ell a\leq \frac{d}{2} < \delta+\ell a$ is similar.

For fixed $\ell$ such that $1+\ell a\leq d-\delta< \delta+\ell a,$ let 
\begin{equation*}
j \coloneqq \ell a+2\delta-d
\end{equation*} so that the last $j$ nonzero entries of $S_{\ell a}\mathbf{m_k}$ are located  in positions greater than $d-\delta$ and the first $\delta-j$ nonzero entries are located in positions less than or equal to $d-\delta.$
(Note that $1\leq j \leq \delta-1.)$ Then,
\begin{equation*}
\langle S_{\ell a}\mathbf{m_k}, \mathbf{x^-}\rangle = -q\sum_{n=1}^{\delta-j}\mathbf{m_k}(n)+p\sum_{n=\delta-j+1}^{\delta}\mathbf{m_k}(n) = -\langle S_{\ell a}\mathbf{m_k}, \mathbf{x^+}\rangle + 2p\sum_{n=\delta-j+1}^{\delta}\mathbf{m_k}(n) . 
\end{equation*} 
Therefore, 
\begin{equation}\label{entrywise}
|Z^-_{k,\ell}-Z^+_{k,\ell}|\leq 2jp\|\mathbf{m}\|_\infty.
\end{equation}
Since $1\leq j\leq \delta-1,$ summing over the set of $\ell$ such that $1+\ell a\leq d-\delta< \delta+\ell a,$ 
 corresponds to summing over  $j=a,2a,\ldots, \lfloor\frac{\delta-1}{a}\rfloor a$ if $a$ divides $d-2\delta,$ or summing over $j=j_0,j_0+a,j_0+2a,\ldots,j_0+\lfloor\frac{\delta-j_0-1}{a}\rfloor a$ for some $0<j_0< a$ otherwise.
Therefore, in either case
\begin{equation}\label{Zdist}
\|Z^+-Z^-\|^2_2 \leq C\|\mathbf{m}\|_\infty^2p^2\sum_{k=1}^K\sum_{t=1}^{\lfloor\delta/a\rfloor+1} |at|^2\leq CKa^2\|\mathbf{m}\|_\infty^2p^2\left(\frac{\delta}{a}\right)^3=CKp^2\|\mathbf{m}\|^2_\infty\frac{\delta^3}{a},
\end{equation}
which proves (\ref{Zdiff}) and completes the proof.
\end{proof}
\begin{proof}[Proof of Theorem \ref{maintheorem1}]
Similarly to the proof of Theorem \ref{maintheorem2},

\begin{equation}\label{ratioA}
C_A \geq \sup\frac{d_1(\mathbf{x},\mathbf{x'})}{\|Y(\mathbf{x})-Y(\mathbf{x'})\|_2},
\end{equation}
where the supremum is again taken over all $\mathbf{x}\not\sim\mathbf{x'}\in\mathcal{C}_{p,q}.$ 
Let $\mathbf{x^\pm}$ be in as in the proof of Theorem \ref{maintheorem2}, and let $Y^\pm \coloneqq Y(\mathbf{x^\pm}).$ By the same reasoning as in the previous proof, $Y^+_{k,\ell} =   Y^-_{k,\ell},$ unless $1+\ell a\leq \frac{d}{2} < \delta+\ell a$ or $1+\ell a \leq d-\delta <\delta+\ell a.$ We will again restrict attention to the case where $1+\ell a \leq d-\delta <\delta+\ell a.$ Let $\ell$ be such that $1+\ell a\leq d-\delta <d +\ell a,$ and again let 
$j \coloneqq \ell a+2\delta-d.$

Since for all $k$ and $\ell,$ we have\begin{equation*}|Z^\pm_{k,\ell}|\leq q\|\mathbf{m}\|_\infty\delta,\end{equation*}
 we see
\begin{equation*}
|Y^+_{k,\ell}-Y^-_{k,\ell}|_2 = |(Z^+_{k,\ell})^2-(Z^-_{k,\ell})^2|=|Z^+_{k,\ell}+Z^-_{k,\ell}||Z^+_{k,\ell}-Z^-_{k,\ell}|\leq 4\|\mathbf{m}\|^2_\infty q\delta p j, 
\end{equation*}
by (\ref{entrywise}). 
Therefore,  by the same reasoning as in (\ref{Zdist}),
\begin{equation*}
\|Y^+-Y^-\|_2^2 \leq C \|\mathbf{m}\|^4_\infty q^2\delta^2p^2\sum_{k=1}^K\sum_{t=1}^{\lfloor\delta/a\rfloor+1}  \left| ta\right|^2 \leq CK\|\mathbf{m}\|^4_\infty q^2\delta^2p^2 a^2 \left(\frac{\delta}{a}\right)^3 = CK\|\mathbf{m}\|^4_\infty q^2p^2 \frac{\delta^5}{a}.
\end{equation*}
Thus, 
the proof will follow from (\ref{ratioA}) once we show  $d_1(\mathbf{x^+},\mathbf{x^-})\geq Cdq^2$.

For $n,m \in \mathbb{N},$ 
let $0_{n\times m}$ and $\mathbbm{1}_{n\times m}$ denote the $n\times m$ matrices of all zeros and of all ones respectively. With this notation we see that

\begin{equation*}
\mathbf{x^\pm} = (q\mathbbm{1}_{1\times\eta},p\mathbbm{1}_{1\times\delta},\pm q\mathbbm{1}_{1\times\eta},p\mathbbm{1}_{1\times\delta})^T,
\end{equation*}
 and
\begin{equation*}
\mathbf{x^{\pm}}\mathbf{x^\pm}^*=\begin{pmatrix}
q^2\mathbbm{1}_{\eta\times \eta}&qp\mathbbm{1}_{\eta\times \delta}&\pm q^2\mathbbm{1}_{\eta\times \eta} &qp\mathbbm{1}_{\eta\times \delta}\\
qp\mathbbm{1}_{\delta\times \eta}&p^2\mathbbm{1}_{\delta\times \delta}&\pm qp\mathbbm{1}_{\delta\times \eta}&p^2\mathbbm{1}_{\delta\times \delta}\\
\pm q^2\mathbbm{1}_{\eta\times \eta}&\pm qp\mathbbm{1}_{\eta\times \delta}&q^2\mathbbm{1}_{\eta\times \eta}&\pm qp\mathbbm{1}_{\eta\times \delta}\\
qp\mathbbm{1}_{\delta\times \eta} &p^2\mathbbm{1}_{\delta\times \delta}&\pm qp\mathbbm{1}_{\delta\times \eta}&p^2\mathbbm{1}_{\delta\times \delta}
\end{pmatrix},
\end{equation*}
where $\eta\coloneqq  \frac{d}{2}-\delta.$ 
Therefore,
\begin{equation*}
\mathbf{x^{+}}\mathbf{x^+}^*-\mathbf{x^{-}}\mathbf{x^-}^*=2q\begin{pmatrix}
0_{\eta\times \eta}&0_{\eta\times \delta}& q\mathbbm{1}_{\eta\times \eta} &0_{\eta\times \delta}\\
0_{\delta\times \eta}&0_{\delta\times \delta}& p\mathbbm{1}_{\delta\times \eta}&0_{\delta\times \delta}\\
 q\mathbbm{1}_{\eta\times \eta}& p\mathbbm{1}_{\eta\times \delta}&0_{\eta\times \eta}&p\mathbbm{1}_{\eta\times \delta}\\
0_{\delta\times \eta} &0_{\delta\times \delta}& p\mathbbm{1}_{\delta\times \eta}&0_{\delta\times \delta}
\end{pmatrix}.
\end{equation*}

We will show that the matrix
\begin{equation*}M\coloneqq
\begin{pmatrix}
0_{\eta\times \eta}&0_{\eta\times \delta}& q\mathbbm{1}_{\eta\times \eta} &0_{\eta\times \delta}\\
0_{\delta\times \eta}&0_{\delta\times \delta}& p\mathbbm{1}_{\delta\times \eta}&0_{\delta\times \delta}\\
 q\mathbbm{1}_{\eta\times \eta}& p\mathbbm{1}_{\eta\times \delta}&0_{\eta\times \eta}&p\mathbbm{1}_{\eta\times \delta}\\
0_{\delta\times \eta} &0_{\delta\times \delta}& p\mathbbm{1}_{\delta\times \eta}&0_{\delta\times \delta}
\end{pmatrix}\in\mathbb{R}^{d\times d}
\end{equation*}
has two nonzero singular values given by 
\begin{equation}\label{singvals}
\sigma_1=\sigma_2=\sqrt{\eta^2q^2+2\eta\delta p^2}.
\end{equation} This will imply $d_1(\mathbf{x^+},\mathbf{x^-}) = 4q\sqrt{\eta^2q^2+2\eta\delta p^2}\geq Cdq^2$ as desired.  

Using the fact that $\mathbbm{1}_{m\times n}\mathbbm{1}_{n\times k}=n\mathbbm{1}_{m\times k},$ we see that 
\begin{align*}M^TM&=
\begin{pmatrix}
q\mathbbm{1}_{\eta\times \eta}q\mathbbm{1}_{\eta\times \eta}&q\mathbbm{1}_{\eta\times \eta}p\mathbbm{1}_{\eta\times \delta}& 0_{\eta\times \eta} &q\mathbbm{1}_{\eta\times \eta}p\mathbbm{1}_{\eta\times \delta}\\
p\mathbbm{1}_{\delta\times \eta}q\mathbbm{1}_{\eta\times \eta}&p\mathbbm{1}_{\delta\times \eta}p\mathbbm{1}_{\eta\times \delta}&0_{\delta\times \eta}&p\mathbbm{1}_{\delta\times \eta}p_{\eta\times \delta}\\
0_{\eta\times \eta}& 0_{\eta\times \delta}&q\mathbbm{1}_{\eta\times \eta}q\mathbbm{1}_{\eta\times \eta}+2p\mathbbm{1}_{\eta\times \delta}p\mathbbm{1}_{\delta\times \eta}& 0_{\eta\times \delta}\\
p\mathbbm{1}_{\delta\times \eta}q\mathbbm{1}_{\eta\times \eta} &p\mathbbm{1}_{\delta\times \eta}p\mathbbm{1}_{\eta\times \delta}& 0_{\delta\times \eta}&p\mathbbm{1}_{\delta\times \eta}p\mathbbm{1}_{\eta\times \delta}
\end{pmatrix}\\
&=
\begin{pmatrix}
\eta q^2\mathbbm{1}_{\eta\times \eta}&\eta qp\mathbbm{1}_{\eta\times \delta}& 0_{\eta\times \eta} &\eta qp\mathbbm{1}_{\eta\times \delta}\\
\eta qp\mathbbm{1}_{\delta\times \eta}&\eta p^2\mathbbm{1}_{\delta\times \delta}&0_{\delta\times \eta}&\eta p^2\mathbbm{1}_{\delta\times \delta}\\
0_{\eta\times \eta}& 0_{\eta\times \delta}&\eta q^2\mathbbm{1}_{\eta\times \eta}+2\delta p^2\mathbbm{1}_{\eta\times \eta}& 0_{\eta\times \delta}\\
\eta qp\mathbbm{1}_{\delta\times \eta} &\eta p^2\mathbbm{1}_{\delta\times \delta}& 0_{\delta\times \eta}&\eta p^2\mathbbm{1}_{\delta\times \delta}
\end{pmatrix}.
\end{align*}
Therefore, $M^TM$ has rank at most two because the second block of rows is equal to the fourth block of rows, which in turn is a multiple of the first block of rows. (Each block can of course have at most one linearly independent row.) We may check that two linearly independent eigenvectors are given by \begin{equation*}(0_{1\times\eta},0_{1\times\delta},\mathbbm{1}_{1\times\eta},0_{1\times\delta})^T\end{equation*}
and \begin{equation*} (q\mathbbm{1}_{1\times\eta},p\mathbbm{1}_{1\times\delta},0_{1\times\eta},p\mathbbm{1}_{1\times\delta})^T,\end{equation*} each with eigenvalue $\eta(\eta q^2+2\delta p^2).$ This proves (\ref{singvals}) and therefore completes the proof.
\end{proof}

\section{Examples:  Lower Bounds for Specific Measurement Masks}
In this section, we will see that the estimates of Theorems \ref{maintheorem2} and \ref{maintheorem1} can be improved for specific choices of well-conditioned measurement masks.

\subsection{Windowed Fourier Measurement Masks}
In this subsection, we consider a family of masks $\{\mathbf{m_k}\}_{k=1}^{2\delta-1},$ defined by
\begin{equation}\label{Fouriermask} \mathbf{m_k}(n) \coloneqq \begin{cases} 
      \frac{\mathbbm{e}^{-n/b}}{(2\delta-1)^{1/4}}\mathbbm{e}^{2\pi \mathbbm{i}(k-1)(n-1)/(2\delta-1)}& 1\leq n \leq \delta \\
      0 & \delta<n\leq d 
   \end{cases},
\end{equation} for some fixed parameter $b>4.$ Masks of this form are closely related to those used in ptychographic imaging (see, for example, \cite{iwen2016fpr}, Section 1.3 and the references provided therein). 
In \cite{iwen2016fpr} it was shown that, with this choice of masks, the map $Y,$ restricted to the subset of $\mathbb{C}^d$ where $\mathbf{x}(n)\neq 0$ for all $1\leq n \leq d,$ can be inverted by  an algorithm which is both efficient and numerically stable in the case where $L=d.$ 

\begin{corollary}\label{Fouriercorr2}
Let $0<p\leq q,$ and consider the map $Z$ restricted to the subset $\mathcal{C}_{p,q}\subset\mathbb{C}^d/\sim.$ 
 Assume that $\delta\leq\frac{d}{4}$ and that $d=aL$ for some integer $a<\delta.$ Then if  $\{\mathbf{m_k}\}_{k=1}^{2\delta-1}$ is the family of masks given by (\ref{Fouriermask}) and $B$ is any Lipschitz map (with respect to $d_1$) such $B(Z(\mathbf{x}))=\mathbf{x}$ for all $\x\in\mathcal{C}_{p,q},$ then  
\begin{equation}\label{Fourier2}
C_B\geq CK_b \frac{q\sqrt{da}}{p(2\delta-1)^{1/4}\delta^{1/2}}=CK_b \frac{qd}{p\sqrt{L}(2\delta-1)^{1/4}\delta^{1/2}},
\end{equation}
where 
$K_{b}\coloneqq \mathbbm{e}^{1/b}-1,$ 
 $C_B$ is the Lipschitz constant of $B,$ and $C$ is a universal constant.\end{corollary}

\begin{corollary}\label{Fouriercorr1}
Let $0<p\leq q,$ and consider the map $Y$ restricted to the subset $\mathcal{C}_{p,q}\subset\mathbb{C}^d/\sim.$ 
 Assume that $\delta\leq\frac{d}{4}$ and that $d=aL$ for some integer $a<\delta.$ Then if  $\{\mathbf{m_k}\}_{k=1}^{2\delta-1}$ is the family of masks given by (\ref{Fouriermask}) and $A$ is any Lipschitz map (with respect to $d_1$) such $A(Y(\mathbf{x}))=\mathbf{x}$,  for all $\x\in\mathcal{C}_{p,q},$ then 
\begin{equation}\label{Fourier1}
C_A\geq C K_{b}^2\frac{qd\sqrt{a}}{p\sqrt{\delta}}=C K_{b}^2\frac{qd^{3/2}}{p\sqrt{L}\sqrt{\delta}}, 
\end{equation}
where 
$K_{b}\coloneqq \mathbbm{e}^{1/b}-1,$ 
 $C_A$ is the Lipschitz constant of $A,$ and $C$ is a universal constant.\end{corollary}

\begin{remark} For this choice of masks, $K=2\delta-1$ and $\|\mathbf{m}\|_\infty=\mathbbm{e}^{-1/b}(2\delta-1)^{-1/4}.$ Therefore, the constants obtained in Corollaries \ref{Fouriercorr2} and \ref{Fouriercorr1} have the same asymptotic behavior with respect to $a$ and $d,$ but are larger with respect to $\delta$ than those obtained by directly applying Theorems \ref{maintheorem2} and  \ref{maintheorem1} to this choice of masks. 
\end{remark}

\begin{remark} 
Similar lower bounds can be derived for any choice of masks along the lines of \eqref{Fouriermask} whose nonzero entries have magnitudes that form a truncated geometric progression.
\end{remark}

%
\begin{proof}[Proof of Corollary \ref{Fouriercorr2}] Let $\mathbf{x^\pm}$ and $Z^\pm$  be as in the proofs of Theorems \ref{maintheorem2} and \ref{maintheorem1}.
As before, note that  $Z^+_{k,\ell} =   Z^-_{k,\ell}$ except when either $1+\ell a\leq \frac{d}{2} < \delta+\ell a$ or $1+\ell a \leq d-\delta <\delta+\ell a$. We will again restrict attention to the case where $1+\ell a \leq d-\delta <\delta+\ell a.$ 

Fix $\ell$ such that $1+\ell a\leq d-\delta<\delta+\ell a,$ and as in the proof of the preceding theorems, let $j \coloneqq \ell a+2\delta-d$ so that the last $j$ nonzero entries of $S_{\ell a}\mathbf{m_k}$ are located  in positions greater than $d-\delta$ and the first $\delta-j$ nonzero entries are located in positions less than or equal to $d-\delta.$
We have seen that
\begin{equation*}
Z^{\pm}_{k,\ell} = \left|\pm q \sum_{n=1}^{\delta-j}\mathbf{m_k}(n) + p \sum_{n=\delta-j+1}^{\delta} \mathbf{m_k}(n)\right|.
\end{equation*}
Therefore, 
\begin{equation}\label{pwF}
|Z^-_{k,\ell}-Z^+_{k,\ell}|\leq 2p\left|\sum_{n=\delta-j+1}^{\delta} \mathbf{m_k}(n)\right|\leq 2p\sum_{n=\delta-j+1}^\delta |\mathbf{m_k}(n)|.
\end{equation}
To estimate the above sum, we note that $|\mathbf{m_k}(n)|=(2\delta-1)^{-1/4}s^n,$ where  $s\coloneqq\mathbbm{e}^{-1/b}.$
Since $0<s<1,$
\begin{equation*}
\sum_{n=\delta-j+1}^{\delta} \left|\mathbf{m_k}(n)\right| \leq (2\delta-1)^{-1/4}\sum_{n=1}^\delta s^n \leq (2\delta-1)^{-1/4}\frac{s}{1-s}.
\end{equation*}
For each $1\leq k \leq 2\delta-1,$ there are at most $\frac{\delta}{a}$ choices of $\ell$ such that $1+\ell a\leq d-\delta<\delta+\ell a.$ Therefore,
\begin{align*}
\|Z^+-Z^-\|_2^2 &\leq C(2\delta-1)\frac{\delta}{a}p^2(2\delta-1)^{-1/2}\left(\frac{s}{1-s}\right)^2\\
&=C(2\delta-1)^{1/2}\frac{\delta}{a}p^2\left(\frac{\mathbbm{e}^{-1/b}}{1-\mathbbm{e}^{-1/b}}\right)^2\\
&=C(2\delta-1)^{1/2}\frac{\delta}{a}p^2\left(\frac{1}{\mathbbm{e}^{1/b}-1}\right)^2.
\end{align*}
Recalling that $D_2(\mathbf{x^+},\mathbf{x^-})\geq q\sqrt{d}$ as shown in the proof of Theorem \ref{maintheorem2} and applying (\ref{ratioB}) completes the proof.
\end{proof}
\begin{proof}[Proof of Corollary \ref{Fouriercorr1}]Let $\mathbf{x^\pm}$ and $Y^\pm$  be as in the proofs of Theorems \ref{maintheorem2} and \ref{maintheorem1}.
Note that for all $k,\ell,$ 
\begin{equation}
|Z^\pm_{k,\ell}| \leq q\sum_{n=1}^\delta |\mathbf{m_k}(n)|
\leq q(2\delta-1)^{-1/4}\sum_{n=1}^\delta s^n
\leq q(2\delta-1)^{-1/4}\frac{s}{1-s},\label{pwF2}
\end{equation}
where $s=\mathbbm{e}^{-1/b}$ as in the proof of Corollary \ref{Fouriercorr1}.
We again note that  $Y^+_{k,\ell} =   Y^-_{k,\ell}$ except when either $1+\ell a\leq \frac{d}{2} < \delta+\ell a$ or $1+\ell a \leq d-\delta <\delta+\ell a$ and again restrict attention to the case where $1+\ell a \leq d-\delta <\delta+\ell a.$ 
 Combining (\ref{pwF}) and (\ref{pwF2}) gives 
\begin{align*}
|Y^+_{k,\ell}-Y^-_{k,\ell}| &= |Z^+_{k,\ell}+Z^-_{k,\ell}||Z^+_{k,\ell}-Z^-_{k,\ell}|\\
&\leq Cqp (2\delta-1)^{-1/2}\left(\frac{s}{1-s}\right)^2.
\end{align*}
  For each $1\leq k \leq 2\delta-1,$ there are at most $\frac{\delta}{a}$ choices of $\ell$ such that $1+\ell a\leq d-\delta<\delta+\ell a.$
 Therefore, \begin{align*}
\|Y^+-Y^-\|_2^2 &\leq C (2\delta-1) \frac{\delta}{a}q^2p^2(2\delta-1)^{-1}\left(\frac{s}{1-s}\right)^4\\
&\leq C  \frac{\delta}{a}q^2p^2\left(\frac{\mathbbm{e}^{-1/b}}{1-\mathbbm{e}^{-1/b}}\right)^4\\
&= C  \frac{\delta}{a}q^2p^2\left(\frac{1}{\mathbbm{e}^{1/b}-1}\right)^4.
\end{align*}
Recalling $d_1(\mathbf{x^+},\mathbf{x^-})\geq Cdq^2,$ as shown in the proof of Theorem \ref{maintheorem2}, completes the proof.
\end{proof}

\subsection{Two-Shot Measurement Masks} Consider the family of masks $\{\mathbf{m_k}\}_{k=1}^{2\delta-1}$ defined by \begin{align}
\mathbf{m_1} &\coloneqq  \mathbf{e_1 }\nonumber\\
\label{holodef}\mathbf{m_{2j}}&\coloneqq \mathbf{e_1}+\mathbf{e_{j+1}}\\
\mathbf{m_{2j+1}} &\coloneqq  \mathbf{e_1}+\mathbbm{i}\mathbf{e_{j+1}}\nonumber
\end{align} for $1\leq j \leq \delta-1,$ where $\{\mathbf{e_1},\ldots,\mathbf{e_d}\}$ is the standard orthonormal basis for $\mathbb{R}^d.$  
In \cite{robustPR} it was shown that, with this choice of masks, the map $Y$ is injective on the subset of $\mathbb{C}^d$ where all entries are nonzero and can be inverted through a well-conditioned algorithm in the case $L=d.$

\begin{corollary}\label{holoprop2}Fix $0<p\leq q,$ and consider the map $Z$ restricted to the subset $\mathcal{C}_{p,q}\subset\mathbb{C}^d/\sim.$
Assume that $\delta\leq\frac{d}{4}$ and that $d=aL$ for some integer $a<\delta.$ Then if $\{\mathbf{m_k}\}_{k=1}^{2\delta-1}$ is the family of masks defined by (\ref{holodef}) and  $B$ is any Lipschitz map (with respect to $D_2$)  such that  $B(Z(\mathbf{x}))=\mathbf{x}$ for all $\x\in\mathcal{C}_{p,q},$ then  
 
\begin{equation*}
\label{holo2}C_B\geq C\frac{q\sqrt{da}}{p\delta}=C\frac{qd}{\sqrt{L}p\delta},
\end{equation*}  
where $C_B$ is the Lipschitz constant of $B$ and $C$ is a universal constant.
\end{corollary}

\begin{corollary}\label{holoprop1}
Let $0<p\leq q,$ and consider the map $Y$ restricted to the subset $\mathcal{C}_{p,q}\subset\mathbb{C}^d/\sim.$ Assume that $\delta\leq\frac{d}{4}$ and that $d=aL$ for some integer $a<\delta.$ Then if $\{\mathbf{m_k}\}_{k=1}^{2\delta-1}$ is the family of masks defined by (\ref{holodef}) and $A$ is any Lipschitz map (with respect to $d_1$)  such that  $A(Y(\mathbf{x}))=\mathbf{x}$  for all $\x\in\mathcal{C}_{p,q},$ then
 
\begin{equation*}
\label{holo1}C_A\geq C\frac{qd\sqrt{a}}{p\delta}=C\frac{qd^{3/2}}{\sqrt{L}p\delta},
\end{equation*}  
where $C_A$ is the Lipschitz constant of $A$  and $C$ is a universal constant.
\end{corollary}

\begin{remark}
Note that for this choice of masks $K=2\delta-1.$ Therefore, the constants obtained in Corollaries \ref{holoprop2} and \ref{holoprop1} exhibit the same asympotic behavior with respect to $d$ and are asymptotically larger with respect to $\delta$ than those obtained by applying Theorems $\ref{maintheorem2}$ and \ref{maintheorem1} to this choice of masks.
\end{remark}
\begin{proof}[Proof of Corollary \ref{holoprop2}]
Let $\mathbf{x^\pm}$ be as in the proof of Theorems \ref{maintheorem2} and \ref{maintheorem1}. Note that for all $1\leq n\leq d,$ $|\mathbf{x^+}(n)|=|\mathbf{x^-}(n)|.$ Therefore, it is clear that for all $\ell,$
\begin{equation*}
|\langle S_{\ell a}\mathbf{m_1},\mathbf{x^+}\rangle | =|\mathbf{x^+}(\ell a+1)|= |\mathbf{x^-}(\ell a+1)|=|\langle S_{\ell a}\mathbf{m_1},\mathbf{x^-}\rangle|,
\end{equation*}
and
\begin{equation*}
|\langle S_{\ell a}\mathbf{m_{2j+1}},\mathbf{x^+}\rangle | = |\mathbf{x^+}(\ell a+1)+\mathbbm{i}\mathbf{x^+}(\ell a+j+1)| = |\mathbf{x^-}(\ell a+1)+\mathbbm{i}\mathbf{x^-}(\ell a+j+1)|=\langle S_{\ell a}\mathbf{m_{2j+1}},\mathbf{x^-}\rangle|
\end{equation*}
since the real and imaginary parts of $\langle S_{\ell a}\mathbf{m_{2j+1}},\mathbf{x^+}\rangle $ and  $\langle S_{\ell a}\mathbf{m_{2j+1}},\mathbf{x^-}\rangle$ have the same absolute values.
Therefore, to estimate $\|Z^+-Z^-\|_2$ we only need to consider the terms  $Z^+_{2j,\ell}-Z^-_{2j,\ell}.$ Furthermore, it is clear that $Z^+_{2j,\ell}$ will equal $Z^-_{2j,\ell},$ unless $\ell$ is chosen in such a way that either $\ell a+1\leq \frac{d}{2} < \ell a+j+1$ or  $\ell a+1\leq d-\delta < \ell a+j+1.$ In either of these cases,

\begin{equation}\label{eachterm}
|Z^{+}_{2j,\ell} -Z^{-}_{2j,\ell}|= 2p.
\end{equation} 
Therefore, we will be able to compute $\|Z^+-Z^-\|_2$ once we estimate the number of $\ell$ such that  $\ell a+1\leq \frac{d}{2} < \ell a+j+1$ or  $\ell a+1\leq d-\delta < \ell a+j+1,$ which we will do in the following lemma.
\begin{lem}\label{holocount}
For fixed $j,$ the number of $\ell$ such that $\ell a+1\leq \frac{d}{2} < \ell a+j+1$ is less than or equal to $\frac{j}{a}.$  Likewise, the number of $\ell$ such that $\ell a+1\leq d-\delta < \ell a+j+1$ is less than or equal to $\frac{j}{a}.$ 
\end{lem}
\begin{proof}
If  $\ell a+1\leq \frac{d}{2} < \ell a+j+1,$ then $\frac{d}{2}-j\leq \ell a \leq \frac{d}{2}-1,$ and any set of $j$ consecutive integers can contain at most $\frac{j}{a}$ multiples of $a.$ Likewise, if $\ell a+1\leq d-\delta < \ell a+j+1,$ then $d-\delta-j\leq \ell a \leq d-\delta-1.$
\end{proof}
Combining (\ref{eachterm}) and Lemma \ref{holocount} gives

\begin{equation*}
\|Z^+-Z^-\|_2^2 \leq \sum_{j=1}^\delta \frac{2j}{a}(2p)^2 \leq C \frac{ p^2\delta^2}{a} = C \frac{ Lp^2\delta^2}{d}.
\end{equation*}
Therefore, recalling  the fact that $D_2(\mathbf{x^+},\mathbf{x^-})\geq \sqrt{d}q,$ as shown in the proof of Theorem \ref{maintheorem2}, the proof follows from  (\ref{ratioB}).
\end{proof}
\begin{proof}[Proof of Corollary \ref{holoprop1}]
Since each $\mathbf{m_k}$ has at most two nonzero entries,  $|Z^+_{k,\ell}+Z^-_{k,\ell}|\leq 4q$ for all $k$ and $\ell.$ Therefore, by (\ref{eachterm}) each nonzero entry of $Y^+-Y^{-}$ satisfies
\begin{equation*}|Y^+_{k,\ell}-Y^-_{k,\ell}| \leq |Z^+_{k,\ell}+Z^-_{k,\ell}||Z^+_{k,\ell}-Z^-_{k,\ell}|
\leq Cqp. 
\end{equation*} Furthermore, similarly to the proof of Corollary \ref{holoprop2}, $Y^+_{k,\ell}-Y^-_{k,\ell}$ is nonzero if and only if $k=2j$ for some $1\leq j\leq \delta-1$ and $\ell a+1\leq \frac{d}{2} < \ell a+j+1$ or $\ell a+1\leq d-\delta < \ell a+j+1.$  Therefore, by Lemma \ref{holocount},
\begin{equation*}
\|Y^+-Y^-\|^2_2 \leq C\sum_{j=1}^\delta \frac{2j}{a}(p q)^2 \leq C \frac{ q^2p^2\delta^2}{a}=C \frac{ Lq^2p^2\delta^2}{d}.
\end{equation*}
Finally, recalling from the proof of Theorem \ref{maintheorem1} that $d_1(\mathbf{x^+},\mathbf{x^-})\geq Cdq^2,$  the result follows from (\ref{ratioA}).
\end{proof}

\section{Discussion and Future Work}

We believe that this initial work opens up several interesting corridors for further research.  First and perhaps most obvious among these is the development of algorithms together with optimal STFT windows, etc., that have Lipschitz upper bounds which match these lower bounds to the extent possible (keeping in mind, of course, that the lower bounds developed here may be gross underestimates).  Existing algorithms for local correlation measurements such as \cite{iwen2016fpr,robustPR} yield upper bounds for the measurements $Y$ considered above \eqref{equ:Ymap} with respect to the $D_2$-metric, a metric with respect to which an inverse of $Y$ will not generally be Lipschitz \cite{Balan2016}.  As a result, the upper bounds they provide are not quite appropriate to compare to the lower bounds considered here.  Nonetheless, the Lipschitz lower bounds developed here do seem to at least heuristically justify the necessity of, e.g., the $d$-dependence present in those existing worst case upper bounds. 

Another interesting avenue of research would be to explore the extension of the related infinite-dimensional results developed by Alaifari et al. \cite{Alaifari2018,Grohs2018} to the finite-dimensional discrete setting.  The resulting theory would potentially provide more fine-grained insights into the recovery of samples $\mathbf{x}$ from discrete STFT magnitude measurements, and could also possibly be extended to results concerning general local correlation measurement maps of the type we consider here 
 in a way that would allow for the relaxation of the support assumptions currently made on the masks $\{\mathbf{m_1},\mathbf{m_2},\ldots,\mathbf{m_K}\}$.  
Finally, one could also consider  local Lipschitz and H\"older lower bounds as opposed to global lower bounds.  Though perhaps  more difficult to analyze, such lower bounds may be more likely to correspond to achievable upper bounds.

\bibliographystyle{abbrv}
\bibliography{refs,continuousPRone_IEEE,LowerB_Intro}

\begin{thebibliography}{10}

\bibitem{Alaifari2018}
R.~{Alaifari}, I.~{Daubechies}, P.~{Grohs}, and R.~{Yin}.
\newblock {Stable Phase Retrieval in Infinite Dimensions}.
\newblock {\em Foundations of Computational Mathematics}, 2018.

\bibitem{balan2016frames}
R.~Balan.
\newblock Frames and phaseless reconstruction.
\newblock {\em Finite Frame Theory: A Complete Introduction to
  Overcompleteness}, 93:175, 2016.

\bibitem{Balan2016}
R.~Balan and D.~Zou.
\newblock On {L}ipschitz analysis and {L}ipschitz synthesis for the phase
  retrieval problem.
\newblock {\em Linear Algebra and its Applications}, 496(Supplement C):152 --
  181, 2016.

\bibitem{bandeira2014phase}
A.~S. Bandeira, Y.~Chen, and D.~G. Mixon.
\newblock Phase retrieval from power spectra of masked signals.
\newblock {\em Information and Inference: a Journal of the IMA}, 3(2):83--102,
  2014.

\bibitem{bendory2018non}
T.~Bendory, Y.~C. Eldar, and N.~Boumal.
\newblock Non-convex phase retrieval from {STFT} measurements.
\newblock {\em IEEE Transactions on Information Theory}, 64(1):467--484, 2018.

\bibitem{candes2015phase}
E.~J. Candes, Y.~C. Eldar, T.~Strohmer, and V.~Voroninski.
\newblock Phase retrieval via matrix completion.
\newblock {\em SIAM review}, 57(2):225--251, 2015.

\bibitem{Candes2014WF}
E.~J. Candes, X.~Li, and M.~Soltanolkotabi.
\newblock Phase retrieval from coded diffraction patterns.
\newblock {\em Applied and Computational Harmonic Analysis}, 39(2):277--299,
  Sept. 2015.

\bibitem{Corbett2006}
J.~Corbett.
\newblock The {P}auli problem, state reconstruction and quantum-real numbers.
\newblock {\em Reports on Mathematical Physics}, 57(1):53--68, 2006.

\bibitem{Fienup1987}
C.~Fienup and J.~Dainty.
\newblock Phase retrieval and image reconstruction for astronomy.
\newblock {\em Image Recovery: Theory and Application}, pages 231--275, 1987.

\bibitem{Grohs2018}
P.~{Grohs} and M.~{Rathmair}.
\newblock {Stable Gabor Phase Retrieval and Spectral Clustering}.
\newblock {\em Communications on Pure and Applied Mathematics}, 2018.

\bibitem{gross2017improved}
D.~Gross, F.~Krahmer, and R.~Kueng.
\newblock Improved recovery guarantees for phase retrieval from coded
  diffraction patterns.
\newblock {\em Applied and Computational Harmonic Analysis}, 42(1):37--64,
  2017.

\bibitem{IARPA_BAA}
IARPA.
\newblock Rapid {A}nalysis of {V}arious {E}merging {N}anoelectronics ({RAVEN}).
\newblock
  \url{https://www.iarpa.gov/index.php/research-programs/raven/raven-baa},
  2016.

\bibitem{iwen2016fpr}
M.~Iwen, A.~Viswanathan, and Y.~Wang.
\newblock Fast phase retrieval from local correlation measurements.
\newblock {\em SIAM Journal on Imaging Sciences}, 9(4):1655--1688, 2016.

\bibitem{robustPR}
M.~A. Iwen, B.~Preskitt, R.~Saab, and A.~Viswanathan.
\newblock Phase retrieval from local measurements: {I}mproved robustness via
  eigenvector-based angular synchronization.
\newblock {\em Applied and Computational Harmonic Analysis}, 2018.

\bibitem{jaganathan2016stft}
K.~Jaganathan, Y.~C. Eldar, and B.~Hassibi.
\newblock {STFT} phase retrieval: {U}niqueness guarantees and recovery
  algorithms.
\newblock {\em IEEE J. Sel. Topics Signal Process.}, 10(4):770--781, 2016.

\bibitem{nawab1983signal}
S.~Nawab, T.~Quatieri, and J.~Lim.
\newblock Signal reconstruction from short-time {F}ourier transform magnitude.
\newblock {\em IEEE Trans. Acoust., Speech, Signal Process.}, 31(4):986--998,
  1983.

\bibitem{pfander2016robust}
G.~E. Pfander and P.~Salanevich.
\newblock Robust phase retrieval algorithm for time-frequency structured
  measurements.
\newblock 2016.
\newblock preprint, arXiv:1611.02540.

\bibitem{Rodenburg2008}
J.~Rodenburg.
\newblock Ptychography and related diffractive imaging methods.
\newblock {\em Advances in Imaging and Electron Physics}, 150:87--184, 2008.

\bibitem{salanevich2015polarization}
P.~Salanevich and G.~E. Pfander.
\newblock Polarization based phase retrieval for time-frequency structured
  measurements.
\newblock In {\em Proc. 2015 Int. Conf. Sampling Theory and Applications
  (SampTA)}, pages 187--191, 2015.

\bibitem{sturmel2011signal}
N.~Sturmel and L.~Daudet.
\newblock Signal reconstruction from {STFT} magnitude: {A} state of the art.
\newblock In {\em Int. Conf. Digital Audio Effects (DAFx)}, pages 375--386,
  2011.

\bibitem{Walther1963}
A.~Walther.
\newblock The question of phase retrieval in optics.
\newblock {\em Journal of Modern Optics}, 10(1):41--49, 1963.

\bibitem{zheng2013wide}
G.~Zheng, R.~Horstmeyer, and C.~Yang.
\newblock Wide-field, high-resolution fourier ptychographic microscopy.
\newblock {\em Nature photonics}, 7(9):739, 2013.

\end{thebibliography}

\end{document}